\documentclass[11pt,a4paper,english]{amsart}

\usepackage[english]{babel}
\usepackage[T1]{fontenc}
\usepackage[utf8]{inputenc}
\usepackage{amsmath}
\usepackage{amssymb}
\usepackage{amsthm}
\usepackage{latexsym}
\usepackage{amsfonts}

\usepackage[onehalfspacing]{setspace} 
\setdisplayskipstretch{.421} 
\newlength{\abovebis} 
\newlength{\belowbis} 
\newlength{\aboveshortbis} 
\newlength{\belowshortbis} 
\everydisplay\expandafter{%
  \the\everydisplay 
  \advance\abovedisplayskip\abovebis 
  \advance\belowdisplayskip\belowbis 
  \advance\abovedisplayshortskip\aboveshortbis 
  \advance\belowdisplayshortskip\belowshortbis 
} 

\def\real{\mathbb{R}}
\def\set#1{\{ \; #1 \; \}}

\def\compl{\mathbb{C}}

\theoremstyle{plain}

\newtheorem{lem}{Lemma}
\newtheorem{theo}[lem]{Theorem}
\newtheorem{prop}[lem]{Proposition}

\theoremstyle{definition}

\newtheorem{rem}{Remark}

\numberwithin{equation}{section}

\begin{document}
\title[Anisotropic inverse problem]{On an inverse problem for anisotropic conductivity in the plane}
\author{Gennadi Henkin}
\address[G. Henkin]{Universit\'{e} Pierre et Marie Curie\\
case 247, 4, place Jussieu, 75252, Paris, France}
\email{henkin@math.jussieu.fr}
\author{Matteo Santacesaria}
\address[M. Santacesaria]{Centre de Mathématiques Appliquées, \'Ecole Polytechnique, 91128, Palaiseau, France}
\email{santacesaria@cmap.polytechnique.fr}
\begin{abstract}
Let $\hat \Omega \subset \real^2$ be a bounded domain with smooth boundary and $\hat \sigma$ a smooth anisotropic conductivity on $\hat \Omega$. Starting from the Dirichlet-to-Neumann operator $\Lambda_{\hat \sigma}$ on $\partial \hat \Omega$, we give an explicit procedure to find a unique (up to a biholomorphism) domain $\Omega$, an isotropic conductivity $\sigma$ on $\Omega$ and the boundary values of a quasiconformal diffeomorphism $F:\hat \Omega \to \Omega$ which transforms $\hat \sigma$ into $\sigma$.
\end{abstract}

\maketitle
\section{Introduction}
Let $\Omega \subset \real^2$ be a bounded domain, and let $\sigma$ be a $C^2$-anisotropic conductivity defined over $\Omega$, i.e. $\sigma = (\sigma^{ij})$ is a positive definite symmetric matrix on $\overline \Omega$ in the $C^2$ class. The corresponding Dirichlet-to-Neumann map is the operator $\Lambda_{\sigma} : C^1(\partial \Omega) \to L^p(\partial \Omega)$, $p < \infty$ defined by
\begin{equation}
\Lambda_{\sigma} f = \sigma \frac{\partial u}{\partial \nu}|_{\partial \Omega}  ,
\end{equation}
where $f \in C^1(\partial \Omega)$, $\nu$ is the outer normal of $\partial \Omega$, and $u$ is the $C^1(\overline \Omega)$-solution of the Dirichlet problem
\begin{equation}
\nabla \cdot (\sigma \nabla u) = 0 \; \textrm{on} \; \Omega, \; \; \; u|_{\partial \Omega}=f.
\label{cons}
\end{equation}

The equation (\ref{cons}) represents the conservation of the electrical charge on $\Omega$ if the voltage potential $f$ is applied to $\partial \Omega$, and $\Lambda_{\sigma} f$ is the current flux at the boundary.
The following inverse problem arises from this construction: how much information about $\sigma$ can be detected from the knowledge of the mapping $\Lambda_{\sigma}$? 

Inverse boundary values problems of such a type were formulated in precise mathematical terms by I. Gel'fand \cite{Ge} and by A. Calderon \cite{C}. These problems arise naturally in several areas: geophysical electrical
prospecting (L. Slichter \cite{S}, V. Druskin \cite{D}), medical imaging (D. Barber, B. Brown \cite{BB}), nondestructive testing of materials (A. Friedman, M. Vogelius \cite{FV}), etc.

It is not possible to determine $\sigma$ uniquely from $\Lambda_{\sigma}$. This was discovered by L. Tartar (see \cite{Kohn}). Indeed, let $\Phi:\overline \Omega \to \overline \Omega$ be a diffeomorphism with $\Phi|_{\partial \Omega}=\mathrm{Id}$, where $\mathrm{Id}$ is the identity map. Then we can define the push-forward of $\sigma$ as 
$$\Phi_{\ast} \sigma =  \left( \frac{{}^t (D\Phi) \sigma (D\Phi)}{|\det(D\Phi)|}\right) \circ \Phi^{-1},$$
where $D\Phi$ is the matrix differential of $\Phi$, and one verifies that $\Lambda_{\Phi_{\ast}\sigma} = \Lambda_{\sigma}$.
In dimension two this is the only obstruction to unique identifiability of the conductivity. The anisotropic problem can be reduced to the isotropic one by using isothermal coordinates (Sylvester \cite{Sylvester}), and combining this technique with the result of Nachman for isotropic conductivities (\cite{Nachman}) we obtain the uniqueness result for anisotropic conductivities with two derivatives. The optimal regularity condition was later obtained by Astala-Lassas-Paivarinta, who proved the uniqueness for $L^{\infty}$-conductivities in \cite{Astala}: for an anisotropic conductivity $\sigma \in L^{\infty}(\Omega)$ ($\Omega \subset \real^2$ bounded simply connected domain) the Dirichlet-to-Neumann map determines the equivalence class of conductivities $\sigma'$ such that there exists a diffeomorphism $\Phi : \Omega \to \Omega$ in the $W^{1,2}$ class with $\Phi|_{\partial \Omega} = \mathrm{Id}$ and $\sigma' = \Phi_{\ast} \sigma$.

The main purpose of this article is to clarify and show what one can explicitly reconstruct from a given Dirichlet-to-Neumann operator in the anisotropic case. From the results obtained in \cite{Sylvester}, \cite{Novikov}, \cite{Nachman}, \cite{Gutarts} we have deduced

\begin{theo} \label{teo1}
Let $\hat \Omega \subset \real^2$ be a bounded domain with $C^1$ boundary and let $\hat \sigma$ be a $C^2$-anisotropic conductivity on $\hat \Omega$, isotropic in a neighbourhood of $\partial \hat \Omega$. Suppose we know $\Lambda_{\hat \sigma} : C^1(\partial \hat \Omega) \to L^p(\partial \hat \Omega)$, $p <\infty$.

Then we can reconstruct a unique domain $\Omega \subset \real^2 \sim \compl$ (up to a biholomorphism), an isotropic conductivity $\sigma$ on $\Omega$ and the boundary values $F|_{\partial \hat \Omega}$ of a quasiconformal $C^1$-diffeomorphism $F:\hat \Omega \to \Omega$ such that $\sigma = F_{\ast} \hat \sigma$.
\end{theo}

The new point in this statement is the existence of $F: \hat \Omega \to \Omega$ (and its explicit reconstruction at the boundary) without any assumption on the topology of $\hat \Omega$. Early in \cite{Astala} this result was proved for simply connected domains, a situation in which the question about deformations of complex structures of $\hat \Omega$ does not make sense.

\smallskip

Our main tool, as in \cite{Sylvester} and \cite{Astala}, is the global solution $F$ of a certain Beltrami equation equipped with an asymptotic condition, which takes our anisotropic conductivity $\hat \sigma$ into an isotropic one, $\sigma$, defined in general over a different domain $\Omega = F(\hat \Omega)$.
With the help of $F$ we then show the existence and uniqueness of a family of solutions $\hat \psi(z,\lambda)$ of the anisotropic conductivity equation, with special asymptotics at infinity, using also the existence of such type of functions in the isotropic case, that we call $\psi(w,\lambda)$ (firstly introduced by Faddeev in \cite{Faddeev}; see \cite{Novikov}, \cite{Nachman} for the main properties).
Then we show how one can reconstruct the boundary values of $\hat \psi$ from the Dirichlet-to-Neumann operator $\Lambda_{\hat \sigma}$, for any $\lambda$, with a Fredholm-type integral equation, following the work of Gutarts (\cite{Gutarts}). This is a generalization of R. Novikov's method for isotropic conductivities (\cite{Novikov}). We also show how to find the boundary values of $F$ from the knowledge of $\hat \psi|_{\partial \hat \Omega}$ (generalizing the result in \cite{Sylvester} and \cite{Astala}), and so we find $F(\partial \hat \Omega) = \partial \Omega$ (therefore also $\Omega$).

After this, we explain how the knowledge of $\Lambda_{\hat \sigma}, \; \hat \psi|_{\partial \hat \Omega}$ and $F|_{\partial \hat \Omega}$ suffices to reconstruct the isotropic scattering amplitude $b(\lambda)$. We give also another method: we define the anisotropic scattering amplitude $\hat b(\lambda)$, and we show that it is equal to the isotropic one, proving that it is essentially a quasiconformal invariant. This result was already included in \cite{Gutarts}; here we give a new simpler proof.

Thus with both methods, starting from $\hat \psi|_{\partial \hat \Omega}$ we can reconstruct the isotropic scattering amplitude: this allows us to write the $\overline \partial$-equation which will permit us to find the isotropic conductivity $\sigma$ on $\Omega$, by the Novikov-Nachman reconstruction scheme (\cite{Novikov}, \cite{Nachman}).

Our scheme can be summarized in the following diagram

$$\Lambda_{\hat \sigma} \to \hat \psi|_{\partial \hat \Omega} \to \left\{ \begin{array}{c} b(\lambda) \\ F|_{\partial \hat \Omega} \end{array} \right. \to \left\{ \begin{array}{c} \sigma \\ \Omega \end{array} \right.$$

All steps of this reconstruction scheme are explicit and can be numerically implemented using the Novikov-Nachman reconstruction-type algorithm \cite{Novikov}, \cite{Nachman}. Therefore, our paper admits potential practical applications.

\begin{rem}
Although we cannot reconstruct $\hat \sigma$ uniquely, for the applications it may be useful to find one representative of the equivalence class of $\hat \sigma$. To do this, using our theorem it suffices to find a diffeomorphism $G: \hat \Omega \to  \Omega$ with fixed boundary values (which are the boundary values of a quasiconformal mapping, in our notation $F|_{\partial \hat \Omega}$), and no other particular restriction: in this way $(G^{-1})_{\ast} \sigma$ will be a representative of $\hat \sigma$. If $\Omega$ is simply connected one can use the Ahlfors-Beurling extension theorem for quasi-symmetric homeomorphism of the circle (\cite[Thm. 2, p.69]{Ahlfors}).
\end{rem}

\begin{rem}
An analogous result to our Theorem 1 is valid also on bordered surfaces in $\real^3$.
\end{rem}

\begin{rem}
One of the referees has drown our attention to the possible relation of our paper to the publications \cite{Kol1} and \cite{Kol2}. In these papers is shown that for the inverse isotropic-conductivity problem in an inaccurately modelled (simply connected) domain there is a unique anisotropic conductivity, corresponding to the boundary measurements, which has the minimal possible anisotropy; this minimally-anisotropic conductivity can be <<isotropized>>, using Beltrami equation, in order to obtain the original isotropic conductivity (up to biholomorphisms of simply connected domains). These papers have certainly some common parts with \cite{Astala}, where the inverse anisotropic-conductivity problem on simply connected domains is studied. But these publications have no common points with our paper; indeed our main novelty consists in the complete study of the inverse anisotropic-conductivity problem in arbitrary domains (not necessarily simply connected) with smooth boundaries. Nevertheless, our results can be applied to extend the above-mentioned publications to the case of non simply connected domains.
\end{rem}

\section{The Beltrami equation and Faddeev-type anisotropic solutions}

We identify $\real^2$ with $\compl$ by the map $(x,y) \mapsto x+iy = z$ and we use the notation

$$\partial_z = \frac 1 2 (\partial_x -i \partial_y), \; \; \; \partial_{\overline z} = \frac 1 2 (\partial_x +i\partial_y)$$
where $\partial_x=\partial / \partial x$ and $\partial_y=\partial / \partial y$. We will also use the differential operators $\partial, \; \overline \partial$ such that $\partial f= \partial_z f dz$, $\overline \partial f= \partial_{\overline z} f d \overline z$, with $dz = dx + i dy$, $d\overline z= dx -i dy$. We also recall the identity $d = \partial + \overline \partial$.

We can suppose that $\hat \sigma$, already isotropic near $\partial \hat \Omega$, is the identity near $\partial \hat \Omega$ (see \cite{Nachman} for the reduction to this case). Besides, we extend $\hat \sigma$ to the whole complex plane by putting $\hat \sigma = I$ for $z \in \compl \setminus \hat \Omega$. Then, for the conductivity $\hat \sigma = \hat \sigma^{ij}$ we define the following Beltrami coefficient

$$\mu_1(z) = \frac{-\hat \sigma^{11}(z)+\hat \sigma^{22}(z)-2i \hat \sigma^{12}(z) }{\hat \sigma^{11}(z) + \hat \sigma^{22}(z) + 2 \sqrt{\det(\hat \sigma)}} $$
which satisfies $| \mu_1 (z) | \leq k < 1$ and is compactly supported in $\hat \Omega$. We now recall the existence of a diffeomorphism that transforms $\hat \sigma$ into an isotropic conductivity.

\begin{prop}(Sylvester \cite[Prop. 2.1]{Sylvester})
There is a quasiconformal $C^1$-diffeomorphism \\ $F: \compl \to \compl$ such that

$$F(z) = z + O\left(\frac 1 z\right) \; \; \; as \; |z| \to \infty,$$
and for which
$$(F_{\ast} \hat \sigma)(z) = \sigma (z) I : = (\det ( \hat \sigma))^{1/2} \circ F^{-1}(z) I.$$
\end{prop}

Thanks to results by Ahlfors and Vekua (\cite{Ahlfors}, \cite{Vekua}), $F$ is obtained as the solution of the Beltrami equation $\partial_{\overline z}F=\mu_1 \partial_z F$, so $F$ is holomorphic in $\compl \setminus \hat \Omega$.

\begin{prop} \label{propfad}
There exist unique Faddeev-type solutions of the anisotropic conductivity equation, i.e. functions $\hat \psi (z,\lambda)$ such that 
\begin{equation} \label{anis}
\nabla \cdot (\hat \sigma (\nabla \hat \psi)) = 0
\end{equation}
for all $z \in \compl, \; \; \lambda \in \compl$, and $\hat \psi (z,\lambda) = e^{\lambda z} ( 1+ O(\frac 1 z))$ when $z \to \infty$.
\end{prop}

Proposition \ref{propfad} for the case $\det\hat\sigma$ close to a constant was obtained firstly in \cite{Sylvester}.

\begin{proof}
We define $\Omega = F(\hat \Omega)$ and $q = \frac{\Delta \sigma ^{1/2}}{\sigma^{1/2}}$. It is known that if $u$ is a solution of $\nabla \cdot (\sigma \nabla u) =0$ in $\Omega$, then $\tilde u = \sigma^{1/2} u$ is a solution of 
\begin{equation} \label{sch}
- \Delta \tilde u + q \tilde u =0
\end{equation}
in $\Omega$.
From \cite{Beals}, \cite{Nachman} and \cite{Novikov}, we have that for every $\lambda \in \compl$ there is a unique solution $\tilde \psi (w,\lambda)$ of (\ref{sch}) with the asymptotic behaviour $\tilde \psi (w,\lambda) = e^{\lambda w} ( 1+ O(\frac 1 w))$ when $w \to \infty$.
So we directly have that $\psi(w, \lambda) : = \sigma^{-1/2}\tilde \psi (w, \lambda)$ is a solution of $\nabla \cdot (\sigma \nabla \psi) =0$ with the same asymptotic (because $\sigma = 1$ outside $\Omega$).

Now let $\hat \psi (z, \lambda)$ be a Faddeev-type anisotropic solution. If we consider $\psi' (w, \lambda) = \hat \psi (F^{-1} (w),\lambda)$, we have that $\nabla \cdot (\sigma \nabla \psi') =0$ from the construction of $\sigma$. Using the properties of $F$ and $\hat \psi$, we get, for $w \to \infty$,
\begin{align*}
\psi'(w, \lambda)= \hat \psi (F^{-1}(w),\lambda) &= e^{\lambda F^{-1}(w)} \left(1+O\left(\frac{1}{|F^{-1}(w)|} \right) \right) \\
&=e^{\lambda w} \left(  1+ O\left(\frac{1}{1+|w|}\right) \right)
\end{align*}
showing that $\psi'(w,\lambda)$ satisfies the same asymptotic of $\psi(w, \lambda)$. From the uniqueness of $\psi(w, \lambda)$ we obtain 
\begin{equation} \label{fadd}
\hat \psi (z, \lambda) = \psi (F(z),\lambda),
\end{equation}
which proves both existence and uniqueness.
\end{proof}

From the equality (\ref{fadd}) we can also derive a useful formula to calculate $F|_{\partial \hat \Omega}$.
In fact, results in \cite{Faddeev2} also indicate how the family of Faddeev-type solutions behaves with respect to $\lambda$. We have indeed $|e^{-w \lambda} \tilde \psi (w,\lambda)-1| \to 0$ as $|\lambda| \to \infty$ for every fixed $w \in \compl$. If we take $w \in \compl \setminus \Omega$ the same limit is also valid for $\psi(w,\lambda)$; combining this with (\ref{fadd}) we deduce the following formula.
\begin{prop}(\cite[Prop. 2.7]{Sylvester})
For all $z \in \compl \setminus \hat \Omega$ (in particular for $z \in \partial \hat \Omega$) we have
$$\lim_{|\lambda| \to \infty} \frac{\log (\hat \psi (z,\lambda))}{\lambda}=\lim_{|\lambda| \to \infty} \frac{\log (\psi (F(z),\lambda))}{\lambda} = F(z).$$
\end{prop}

\section{An integral equation for $\hat \psi|_{\partial \hat \Omega}$}
Following the approach of \cite{Gutarts}, we show that, as in the isotropic case, we can find $\hat \psi|_{\partial \hat \Omega}$ through a Fredholm-type integral equation.

The main idea is to decompose the differential operator $-\nabla \cdot \hat \sigma \nabla$ as \linebreak $- \Delta + M,$ where $M$ is a compactly supported operator.
So we can characterize $\hat \psi (z,\lambda)$ as the solution of the following integral equation:

$$\hat \psi (z,\lambda) = e^{z \lambda} - \frac{i}{2}\int_{\hat \Omega} G(z-w,\lambda) M\hat \psi(w,\lambda) dw \wedge d\overline w,$$
where  
$$G(z,\lambda)= \frac{ie^{\lambda z}}{2(2 \pi)^2} \int_{\compl} \frac{e^{i(w \overline z + \overline w z)} dw \wedge d \overline w}{w(\overline w-i \lambda)}, \; \; z \in \compl, \; \lambda \in \compl$$
is the Faddeev-Green function for the Laplacian.

\begin{prop}(\cite[Lemma 2.4]{Gutarts})
For every $\lambda \in \compl$ the boundary value of $\hat \psi$ satisfies
\begin{equation} \label{bord}
\hat \psi (z, \lambda)|_{\partial \hat \Omega} = e^{z \lambda} - \int_{\partial \hat \Omega} G(z-w,\lambda) (\Lambda_{\hat \sigma}-\Lambda_0)\hat \psi(w,\lambda) dw,
\end{equation}
where $\Lambda_0$ is the Dirichlet-to-Neumann operator of the standard Laplacian (or for the case of constant conductivity).
\end{prop}
This follows from the identity
\begin{equation} \label{iden}
\int_{\partial \hat \Omega} u_0(\Lambda_{\hat \sigma} - \Lambda_0)u= \int_{\hat \Omega} u_0 M u,
\end{equation}
where $u_0, u \in W^{1,2}(\hat \Omega), \; \nabla \cdot (\hat \sigma \nabla u) = 0, \; \Delta u_0 =0$ in $\hat \Omega$.

The fact that the integral equation (\ref{bord}) is of Fredholm type in the Sobolev space $W^{s,2}(\partial \hat \Omega)$ is the content of \cite[Lemma 2.5]{Gutarts}, and it is uniquely solvable by \cite[Lemma 2.6]{Gutarts} (these properties are implied by the same results in the isotropic case \cite{Nachman}).

\section{Reconstruction of the scattering amplitude}

Following \cite{Faddeev2}, we define the \textit{non-physical} scattering amplitude for the isotropic inverse problem as
\begin{equation}
b(\lambda) = \int_{\Omega} e^{-\overline \lambda \overline w} q(w)\tilde \psi (w, \lambda) dw.
\end{equation}
From \cite{Novikov} we have

$$b(\lambda) = \int_{\partial \Omega} e^{-\overline \lambda \overline w} (\Lambda_q - \Lambda_0) \tilde \psi (w, \lambda) dw,$$
where $\Lambda_q$ is the Dirichlet-to-Neumann operator of the Schr\"odinger equation (\ref{sch}).

Since $\sigma$ is the identity near $\partial \Omega$, equation $\Lambda_q= \sigma^{-1/2} (\Lambda_{\sigma} + \frac 1 2 \frac{\partial \sigma}{\partial \nu}) \sigma^{-1/2}$ reads $\Lambda_q = \Lambda_{\sigma}$, and $\tilde \psi|_{\partial \Omega} = \psi|_{\partial \Omega}$, so

\begin{equation} \label{scatt1}
b(\lambda) = \int_{\partial \Omega} e^{-\overline \lambda \overline w} (\Lambda_{\sigma} - \Lambda_0) \psi (w, \lambda) dw.
\end{equation}

Thus, for the reconstruction of $b$, it is sufficient to determine $\Lambda_{\sigma}$ and $\psi|_{\partial \Omega}$. By (\ref{fadd}) we already know $\psi|_{\partial \Omega}$; for the determination of $\Lambda_{\sigma}$, by arguments of \cite{Astala}, we obtain the identity 
\begin{equation} \label{diri}
\int_{\partial \hat \Omega} \hat u \Lambda_{\hat \sigma} \hat v =\int_{\partial \Omega} u \Lambda_{\sigma} v
\end{equation}
which holds for any $\hat u, \hat v \in C^{1} (\partial \hat \Omega)$ and $u,v \in C^{1} (\partial \Omega)$ such that $\hat u = u \circ F$ and $\hat v = v \circ F$ (this follows directly from the properties of $F$ and the symmetry of the two Dirichlet-to-Neumann operators). So we find $\Lambda_{\sigma}$ from $\Lambda_{\hat \sigma}$ and $F|_{\partial \hat \Omega}$.

\subsection{Complementary result}
We give here another method to find $b(\lambda)$.
Inspired by \cite{Gutarts}, we define the anisotropic scattering amplitude as
\begin{equation}
\hat b(\lambda) = \int_{\hat \Omega} e^{-\overline \lambda \overline z} M\hat \psi (z, \lambda) dz
\end{equation}

and we have the following result.
\begin{prop} \label{ident}
$b(\lambda) = \hat b(\lambda)$
\end{prop}

We will need the following lemma
\begin{lem}
For every $\phi \in C^1(\partial \hat \Omega)$, $\psi \in C^1(\hat \Omega)$ solution of $\nabla \cdot (\hat \sigma \nabla \psi)=  ( \Delta - M)\psi = 0$ in $\hat \Omega$, we have

\begin{equation}
\int_{\partial \hat \Omega} \phi(\Lambda_{\hat \sigma}-\Lambda_0)\psi = 2i \int_{\partial \hat \Omega} \phi (\overline \partial \psi - \overline \partial \psi_0),
\end{equation}
where $\Delta \psi_0 = 0$ in $\hat \Omega$ and $\psi_0|_{\partial \hat \Omega}=\psi|_{\partial \hat \Omega}$. 
\end{lem}

\begin{proof}
Let $a \in C^1(\hat \Omega)$ such that $a|_{\partial \hat \Omega}=\phi$, and $w=x+iy$. From the definition of the Dirichlet-to-Neumann operator and from Stokes' theorem, one has
$$\int_{\partial \hat \Omega} \phi (\Lambda_{\hat \sigma}-\Lambda_0)\psi = \int_{\hat \Omega} (\nabla a \cdot \nabla (\psi - \psi_0) + a M\psi)dxdy,$$
and by Stokes' theorem and by the identity $\Delta= 4 \frac{\partial^2}{\partial z \partial \overline z}$
$$2i \int_{\partial \hat  \Omega} \phi(\overline \partial \psi - \overline \partial \psi_0) = 2i\int_{\hat \Omega}  \partial a  \wedge (\overline \partial \psi - \overline \partial \psi_0) + \int_{\hat \Omega} a M \psi \; dx dy.$$

Writing in coordinates we get $$\partial a \wedge \overline \partial (\psi-\psi_0) = \frac{1}{2i}\nabla a \cdot \nabla (\psi - \psi_0)dx dy + \frac{1}{2}da \wedge d(\psi-\psi_0).$$
Again by Stokes' thorem we have
$$\int_{\hat \Omega}da \wedge d(\psi - \psi_0) = - \int_{\partial \hat  \Omega} (\psi-\psi_0)da=0$$
because $\psi|_{\partial \hat \Omega} = \psi_0|_{\partial \hat \Omega}$. The proof follows.
\end{proof}

\begin{proof}[Proof of Proposition \ref{ident}]
From identity (\ref{iden}) we find

$$\hat b(\lambda)= \int_{\partial \hat \Omega} e^{-\overline \lambda  \overline z} (\Lambda_{\hat \sigma} - \Lambda_0)\hat \psi(z,\lambda) dz.$$
Using the lemma we find
\begin{align} \label{pd1}
\hat{b}(\lambda) &= 2i \int_{\partial \hat \Omega} e^{- \overline \lambda \overline z} (\overline \partial \hat \psi -\overline \partial \hat \psi_0)= 2i \int_{\partial \hat \Omega} e^{- \overline \lambda \overline z} \overline \partial \hat \psi, \\ \label{pd2}
b(\lambda) &= 2i \int_{\partial \Omega} e^{- \overline \lambda \overline w} (\overline \partial \psi -\overline \partial \psi_0)= 2i \int_{\partial \Omega} e^{- \overline \lambda \overline w} \overline \partial \psi,
\end{align}
where the second equalities follows from Stokes' theorem, the fact that $e^{- \overline \lambda \overline z}$ (resp. $e^{- \overline \lambda \overline w}$) is antiholomorphic and $\hat \psi_0$ (resp. $\psi_0$) is harmonic in $\hat \Omega$ (resp. in $\Omega$).

If we call $z = G(w) =F^{-1}(w)$ we find, from (\ref{pd1}),
\begin{align} \nonumber
\hat b (\lambda) &= 2i \int_{\partial \hat \Omega} e^{- \overline \lambda \overline z}  \frac{\partial \hat \psi}{\partial \overline z} d \overline z \\ \nonumber
&= 2i\int_{\partial \Omega} e^{-\overline{\lambda G(w)}} \overline{\left( \frac{\partial F}{\partial z} \right)} \frac{\partial \psi}{\partial \overline w}(w,\lambda) \overline{\left( \frac{\partial G}{\partial w} \right)} d \overline w\\ \label{pd3}
&= 2i\int_{\partial \Omega} e^{-\overline{\lambda G(w)}} \overline \partial \psi,
\end{align}
because $F$ (resp. $G$) is holomorphic in a neighbourhood of $\partial \hat \Omega$ (resp. $\partial \Omega$), and from the equality $\psi \circ F =\hat \psi$.

To see that (\ref{pd3}) is equal to (\ref{pd2}) we proceed as follows. Let $\Omega_R = \set{ z \in \compl : \; |z| < R}$ the disk of radius $R$, and let $R$ be sufficiently large to have $\overline \Omega \subset \Omega_R$. We apply Stokes' theorem to $\Omega_R \setminus \Omega$ and we obtain, for every quasiconformal homeomorphism $E:\compl \to \compl$, holomorphic in $\compl \setminus \Omega$,
\begin{equation} \int_{\partial \Omega} e^{- \overline{\lambda E(w)}}\overline \partial \psi = \int_{\partial \Omega_R} e^{- \overline{\lambda E(w)}}\overline \partial \psi  + \int_{\Omega_R \setminus \Omega} \partial (e^{- \overline{\lambda E(w)}}\overline \partial \psi) \nonumber
\end{equation} 
but the last term vanishes, because $e^{- \overline{\lambda E(w)}}$ is anti-holomorphic and $\partial \overline \partial \psi =0$ in $\compl \setminus \Omega$.

So the identity
$$\int_{\partial \Omega} e^{- \overline{\lambda E(w)}}\overline \partial \psi = \int_{\partial \Omega_R} e^{- \overline{\lambda E(w)}}\overline \partial \psi$$
is true for $R \gg 0$, $E(w)=G(w)$ and $E(w)=w$. As we have $G(w) = w + O(\frac{1}{|w|})$ for $w \to \infty$, using the lemma we deduce

\begin{align*}
\hat b(\lambda)&=2i \int_{\partial \Omega} e^{-\overline {\lambda G(w)}}\overline \partial \psi = \lim_{R \to \infty} 2i \int_{\partial \Omega_R} e^{- \overline{\lambda G(w)}}\overline \partial \psi \\ 
&=\lim_{R \to \infty}  2i \int_{\partial \Omega_R} e^{- \overline{\lambda w}}\overline \partial \psi = 2i\int_{\partial \Omega} e^{-\overline {\lambda w}}\overline \partial \psi  = b(\lambda) \qedhere
\end{align*}
\end{proof}

\section{The $\overline \partial$-equation and the reconstruction of $\sigma$}
Here we follow the steps of \cite{Novikov} to reconstruct isotropic conductivities.
The function $\mu(w,\lambda)=\tilde \psi (w,\lambda) e^{-\lambda w}$ satisfies the following $\overline \partial$-equation with respect to $\lambda$

\begin{equation}
\frac{\partial \mu(w,\lambda)}{\partial \overline \lambda}=\frac{b(\lambda)}{4\pi \overline \lambda}e^{\overline \lambda \overline w - \lambda w}\overline{\mu(w,\lambda)}.
\end{equation}
This is equivalent to the integral equation:
\begin{equation}\label{ciao}
\mu(w,\lambda) = 1 + \frac{1}{8 \pi^2 i}\int_{\compl}\frac{b(\lambda')}{(\lambda'-\lambda) \overline \lambda'}e^{\overline \lambda' \overline w - \lambda' w}\overline{\mu(w,\lambda')}d\lambda' \wedge d \overline \lambda'
\end{equation}
because $\mu \to 1$ when $w \to \infty$. By results of \cite{Nachman}, equation (\ref{ciao}) is solvable, and one can find $\sigma(w)$ from the integral formula
\begin{equation}
\sigma^{1/2}(w)= \mu(w,0) = 1+ \frac{1}{8 \pi^2 i}\int_{\compl}\frac{b(\lambda)}{|\lambda|^2}e^{\overline \lambda \overline w - \lambda w}\overline{\mu(w,\lambda)}d\lambda \wedge d \overline \lambda, \; \; \forall w \in \compl
\end{equation}
or from the more stable general formula
\begin{equation}
\frac{\Delta \sigma^{1/2}(w)}{\sigma^{1/2}(w)} = \frac{\Delta \tilde \psi(w,\lambda)}{\tilde \psi(w,\lambda)}, \; \; \forall w \in \compl, \; \forall \lambda \in \compl.
\end{equation}

\end{document}